\documentclass[12pt]{amsart}
\usepackage{latexsym}
\usepackage{amsmath}
\usepackage{amsfonts}
\usepackage{amsthm}
\usepackage{amssymb}
\usepackage{ifthen}
\usepackage{enumerate}
\usepackage[T1]{fontenc}
\usepackage{dutchcal}
\usepackage{cite}

\usepackage[mathscr]{eucal}
\newcommand*{\rom}[1]{\expandafter\@slowromancap\romannumeral #1@}

\DeclareFontFamily{U}{mathx}{}
\DeclareFontShape{U}{mathx}{m}{n}{<-> mathx10}{}
\DeclareSymbolFont{mathx}{U}{mathx}{m}{n}
\DeclareMathAccent{\widehat}{0}{mathx}{"70}
\DeclareMathAccent{\widecheck}{0}{mathx}{"71}
\def\eq#1{{\rm(\ref{E#1})}}
\def\Eq#1#2{\ifthenelse{\equal{#1}{*}}
  {\begin{equation*}\begin{aligned}#2\end{aligned}\end{equation*}}
  {\begin{equation}\begin{aligned}\label{E#1}#2\end{aligned}\end{equation}}}

\newcounter{allenv}[section]

\newtheorem{thm}{Theorem}
\newtheorem*{thm*}{Theorem}

\newtheorem*{conj*}{Conjecture}
\newtheorem{coro}{Corollary}
\newtheorem{lemm}{Lemma}
\newtheorem*{lemm*}{Lemma}

\theoremstyle{remark}
\newtheorem*{remark*}{Remark}

\newtheorem*{exmp*}{Example}

\theoremstyle{definition}

\author[]{Tibor Kiss and Dóra Koroknai}
\title[]{Non-symmetrically $t$-affine functions revisited} 
\address{Institute of Mathematics,
University of Debrecen,
4002 Debrecen, Pf.~400, Hungary}
\email{kiss.tibor@science.unideb.hu}
\email{koroknai.dora1@mailbox.unideb.hu}


\thanks{The first author’s research was supported by the HUN-REN Hungarian Research Network.}

\def\eq#1{{\rm(\ref{E#1})}}

\def\Eq#1#2{\ifthenelse{\equal{#1}{*}}
	{\begin{equation*}\begin{aligned}[]#2\end{aligned}\end{equation*}}
	{\begin{equation}\begin{aligned}\label{E#1}#2\end{aligned}\end{equation}}}

\begin{document}

\begin{abstract}
In 2014, Michał Lewicki and Andrzej Olbryś proved that if a real valued function $f$ defined on the real line satisfies the conditional functional equation
\Eq{*}{
f(tx + (1-t)y) = t f(x) + (1-t) f(y),\qquad x\leq y,
}
called non-symmetrically $t$-affine, then it is $t$-affine. That is, they concluded that $f$ must fulfill the above equality without any restriction on $x$ and $y$.

In the current study, first we show that the above conditional equation implies that the function in question is locally $t$-affine. Then we derive $t$-affinity on open intervals. Finally, we formulate our main result, which generalizes the theorem of Lewicki and Olbryś for any subinterval of $\mathbb{R}$.
\end{abstract}

\maketitle

\section{Introduction}

For an interval $I\subseteq\mathbb{R}$ of positive length and $t\in[0,1]$, a function $f:I\to\mathbb{R}$ is called \emph{$t$-convex} if
\Eq{c}{
f(tx+(1-t)y)\leq tf(x)+(1-t)f(y),\qquad x,y\in I.
}\label{eq:c}
Clearly, if $f$ is $t$-convex for every $t\in[0,1]$, then $f$ is \emph{convex} in the classical sense. We shall say that $f$ is \emph{locally $t$-convex on $I$} if, for any point $p\in I$, there exists an open interval $U\subseteq\mathbb{R}$ such that $p\in U$ and $f$ is $t$-convex on $U\cap I$.

We say that $f:I\to\mathbb{R}$ is \emph{$t$-affine} if both of $f$ and $-f$ are $t$-convex. More precisely, if
\Eq{a}{
f(tx + (1-t)y) = t f(x) + (1-t) f(y),\qquad x,y\in I.
}
Accordingly, we say that a function $f$ is \emph{locally $t$-affine} if both $f$ and $-f$ are locally $t$-convex.

In this paper we intend to investigate a conditional version of the above concepts. Namely, we will say that $f:I\to\mathbb{R}$ is \emph{non-symmetrically $t$-convex} if, for all $x,y\in I$, the inequality \eq{c} holds provided that $x\leq y$. It is easy to see that any $t$-convex function is non-symmetrically $t$-convex. In the paper \cite[Example 3.1]{LewOlb14}, the authors proved that the opposite is not necessarily true. To be more precise, they proved that if $t$ is a transcendental number, then there exists a non-symmetrically $t$-convex function, which is not $t$-convex. The existence of such a function, whenever $t$ is algebraic, is still an open question.

Now we introduce the concept of non-symmetrically $t$-affine functions. We say that $f$ is \emph{non-symmetrically $t$-affine} if $f$ and $-f$ are non-symmetrically $t$-convex, that is, if
\Eq{na}{
f(tx + (1-t)y) = t f(x) + (1-t) f(y),\qquad x,y\in I\text{ with }x\leq y.
}

In the same paper \cite{LewOlb14}, Lewicki and Olbryś proved that the conditional equation \eq{na} implies \eq{a} provided that $I=\mathbb{R}$. For the sake of completeness, we state their result.

\begin{thm*}[Lewicki--Olbryś, 2014]
Let $t\in(0,1)$ be fixed. If $f:\mathbb{R}\to\mathbb{R}$ is non-symmetrically $t$-affine then it is $t$-affine.
\end{thm*}

In the proof, the authors make heavy use of the fact that the domain of the function contains 0. Moreover, due to the reflections with respect to $0$, they also need sufficient room on both the positive and the negative side, that is, they require the domain to be the entire real line. The concept of non-symmetrically $t$-affine functions do not require these conditions. Motivated by this, we extend this result to any non-empty open subinterval of $\mathbb{R}$ not necessarily containing zero. Furthermore, we succeeded in substantially simplifying the original proof appeared in \cite{LewOlb14}.

For simplicity, in our paper, $\alpha$ and $\beta$ will denote the infimum and the supremum of $I$, respectively. We emphasize that we do not require the boundedness of $I$. Finally, we note that $(0,1)$ will stand for the open interval $\{s\in\mathbb{R}\colon 0<s<1\}$.

\section{Non-symmetrically $t$-affine functions}

First of all we show that non-symmetric $t$-affine property implies that our function is locally $t$-affine.

\begin{thm}\label{thm1}
Let $I \subseteq \mathbb{R}$ be an open non-empty interval and fix $t \in(0,1)$. If $f:I\to\mathbb{R}$ is non-symmetrically $t$-affine, then it is locally $t$-affine.
\end{thm}

\begin{proof}
If $t=\frac12$, then the statement is trivial. Hence we may and do assume that $t\neq\frac12$. Let $p\in I$ be arbitrary but fixed and let
\Eq{*}{
0 < r < \min \left\{ \frac{\beta - p}{1 + 2t}, \frac{p - \alpha}{3 - 2t} \right\}.}

Since $\max\{\frac{1}{1+2t},\frac{1}{3-2t}\}<1$, it follows that $(p-r,p+r)\subseteq I$ holds, thus $U_p:=(p-r,p+r)$ is an open neighborhood of $p$ in $I$. Now we show that $f$ is $t$-affine on the interval $U_p$.

To do this, let $x,y\in U_p$ with $x\leq y$, and denote $u:=ty+(1-t)x$. We claim that
\Eq{conv}{
f(u)=tf(y)+(1-t)f(x)
}
holds. Obviously, to avoid the trivial cases, we can assume that $x\neq y$. 

Let us reflect the points $x$ and $y$ with respect to $u$ as follows. Define
\Eq{*}{
\lambda_t(x,u):=\frac1{t}(x-(1-t)u)
\qquad\text{and}\qquad
\varrho_t(u,y):=\frac1{1-t}(y-tu).
}
First we show that
\Eq{i}{
\alpha <\lambda_t(x,u)<\varrho_t(u,y)< \beta.
}

Indeed, the inequality $\alpha <\lambda_t(x,u)$ is equivalent to $\alpha<(2-t)x+(t-1)y$, thus focusing on the latter is sufficient. Taking into account the definition of $r$, we obtain that
\Eq{*}{
(2-t)x+(t-1)y>(2-t)(p-r)+(t-1)(p+r)=p+(2t-3)r>p-(p-\alpha)=\alpha,
}
where we also used that $2t-3<0$. The inequality $\varrho_t(u,y)< \beta$ can be verified in a similar way. Finally, inequality $\lambda_t(x,u)<\varrho_t(u,y)$ is equivalent to $2t(1-t)(x-y)<0$, which is true.

Observe that $t\lambda_t(x,u)+(1-t)\varrho_t(u,y)=tx+(1-t)y$. Due to the definition of $\lambda_t(x,u)$ and $\varrho_t(u,y)$, we can write
\Eq{*}{
x=t\lambda_t(x,u)+(1-t)u\qquad\text{and}\qquad
y=tu+(1-t)\varrho_t(u,y).
}
Therefore, we have
\Eq{*}{
f(x)
=f\big(t\lambda_t(x,u)+(1-t)u\big)
=tf\big(\lambda_t(x,u)\big)+(1-t)f(u)
}
and
\Eq{*}{
f(y)=tf(u)+(1-t)f\big(\varrho_t(u,y)\big).
}
Summing up the above equations side by side and considering the chain of inequalities \eq{i}, we obtain that
\Eq{*}{
f(x)+f(y)
&=tf\big(\lambda_t(x,u)\big)+(1-t)f\big(\varrho_t(u,y)\big)+f(u)\\
&=f\big(t\lambda_t(x,u)+(1-t)\varrho_t(u,y)\big)+f(u)\\
&=f(tx+(1-t)y)+f(u)=tf(x)+(1-t)f(y)+f(u).
}
Expressing $f(u)$ and applying the definition of $u$, we get
\Eq{*}{
f(ty+(1-t)x)=tf(y)+(1-t)f(x),
}
which is the desired equality \eq{conv}.
\end{proof}

To extend the $t$-affine property to open intervals, we apply the following localization theorem of Nikodem and Páles appeared originally in \cite[Corollary 6.]{PalNik04}.

\begin{lemm*}[Nikodem--Páles, 2004]
Let $t\in(0,1)$. A function $f:I\to\mathbb{R}$ is $t$-convex on $I$ if and only if, for each point $p\in I$, there exists a neighborhood $U\subseteq I$ such that $f$ is $t$-convex on $I\cap U$.
\end{lemm*}

\begin{thm}\label{M1}
Let $I\subseteq\mathbb{R}$ be an open non-empty interval and fix $t \in (0,1)$. If $f : I \to \mathbb{R}$ is non-symmetrically $t$-affine, then it is $t$-affine.
\end{thm}

\begin{proof}
If $f$ is non-symmetrically $t$-affine, then, in view of Theorem \ref{thm1}, it is locally $t$-affine. Consequently, $f$ and $-f$ are locally $t$-convex on $I$. Thus, by the above Lemma of Nikodem and Páles, it follows that $f$ and $-f$ are $t$-convex on the whole domain $I$. Consequently, $f$ is $t$-convex on $I$.
\end{proof}

As a special case, putting $t = \frac{1}{2}$ into inequality \eqref{eq:c}, we obtain
\begin{equation}
f\left(\frac{x + y}{2}\right) \leq \frac{f(x) + f(y)}{2},
\qquad x,y \in I.
\end{equation}
A function satisfying this inequality is called \emph{Jensen-convex}. If both $f$ and $-f$ are Jensen-convex, that is, $f$ satisfies
\Eq{Je}{f\left(\frac{x + y}{2}\right) = \frac{f(x) + f(y)}{2},\qquad x,y\in I,}
then $f$ is said to be \emph{Jensen-affine}.

As it is well-known (see M. Kuczma \cite{Kuczma}), there is a nice representation of Jensen-affine functions. Namely, if $f:I\to\mathbb{R}$ satisfies equation \eq{Je}, then there uniquely exist an additive function $A:\mathbb{R}\to\mathbb{R}$ and a constant $b\in\mathbb{R}$ such that
\Eq{Aa}{
f(x)=A(x)+b
}
for all $x\in I$. If $f:I\to\mathbb{R}$ is $t$-affine for some $t\in(0,1)$, then, in view of a celebrated result of Kuhn \cite{Kuhn84} or Daróczy and Páles \cite{DarPal87}, $f$ is also Jensen-affine. Consequently, we can clearly associate an additive function that produces $f$ in the \eq{Aa} manner. This function will be called \emph{the additive part of $f$}.

Before we formulate our next lemma, we note that the set $I-p$ denotes the interval $\{s-p\colon s\in I\}$.

\begin{lemm}\label{HomT}
If $f:I\to\mathbb{R}$ is $t$-affine for some $t\in(0,1)$, then its additive part is $t$-homogeneous.
\end{lemm}

\begin{proof}
If $I$ is empty or a singleton, then the statement is trivial. Therefore, we may and do assume that $I$ has a non-empty interior. Let $p\in I^\circ$ be any but fixed and define
\Eq{*}{
f_0:I-p\to\mathbb{R},\qquad
f_0(x):=f(x+p)-f(p).}
Then $f_0(0)=0$ and $f_0$ is $t$-affine as well. Indeed, since $f$ is $t$-affine, for $x,y\in I$, we have
\Eq{*}{
f_0(tx+(1-t)y)
&=f(tx+(1-t)y+p)-f(p)\\
&=f\big(t(x+p)+(1-t)(y+p)\big)-f(p)\\
&=tf(x+p)+(1-t)f(y+p)-f(p)=tf_0(x)+(1-t)f_0(y).}

Using this, for any $x\in I-p$, we can write that
\Eq{*}{
f_0(tx)=f_0(tx+(1-t)\cdot0)=tf_0(x)+(1-t)f_0(0)=tf_0(x),
}
that is, $f_0$ is $t$-homogeneous. This implies $f(tx+p)-f(p)=t\big(f(x+p)-f(p)\big)$ for $x\in I-p$. If $A$ stands for the additive part of $f$, then this latter equality reduces to $A(tx)=tA(x)$. Thus we obtained that there exists a neighborhood $U=I-p$ of $0$ on which $A$ is $t$-homogeneous.

To prove that $A$ is $t$-homogeneous on $\mathbb{R}$, let $v\in\mathbb{R}$ be arbitrary. Then there exists $k\in\mathbb{N}$ such that $\frac{1}{2^k}v\in U$. Then, based on the previous part of the proof, we have $A\big(t\frac{1}{2^k}v\big)=tA\big(\frac{1}{2^k}v\big)$. Since any additive function is $\mathbb{Q}$-homogeneous (cf. \cite{Kuczma}), we obtain that $A(tv)=tA(v)$, which finishes the proof.
\end{proof}

\begin{coro}\label{HomTc}
If $f:I\to\mathbb{R}$ is $t$-affine for some $t\in(0,1)$, then its additive part is $(1-t)$-homogeneous.
\end{coro}

\begin{proof}
Since the set of parameters for which an additive function is homogeneous forms a subfield of $\mathbb{R}$, the above statement follows easily.
\end{proof}

\section{The main result}

\begin{thm}\label{M2}
Let $I\subseteq\mathbb{R}$ be any interval and fix $t \in[0,1]$. If $f : I \to \mathbb{R}$ is non-symmetrically $t$-affine, then it is $t$-affine.
\end{thm}

\begin{proof}
To avoid the trivial cases, we suppose that $I$ is of positive length and $t\notin\{0,\frac12,1\}$.

Then, in view of Theorem \ref{M1}, $f$ is $t$-affine on the interior of $I$, consequently, it is also Jensen-affine. Thus we have \eq{Aa} for all $x\in I^\circ$ for some additive function $A:\mathbb{R}\to\mathbb{R}$ and constant $b\in\mathbb{R}$. 

If $\alpha\in I$, then, for any $u\in I^\circ$, we have that
\Eq{Af}{
A(t\alpha+(1-t)u)+b=tf(\alpha)+(1-t)\big(A(u)+b)=tf(\alpha)+(1-t)A(u)+(1-t)b.
}
Applying the additivity of $A$ and then Lemma \ref{HomT} and Corollary \ref{HomTc} on the left hand side of the above equation, \eq{Af} reduces to $f(\alpha)=A(\alpha)+b$. Now multiplying this with $1-t$ and adding $tA(u)+tb$ to both sides of the equation so obtained, we get the desired equality
\Eq{*}{
f(tu+(1-t)\alpha)=tf(u)+(1-t)f(\alpha)b.
}

If $\beta\in I$ and $u\in I^\circ$, then the computation is completely analogous.

Finally, assume that $\alpha,\beta\in I$. Then the computation in the previous part of the proof yields that $f(\alpha)=A(\alpha)+b$ and $f(\beta)=A(\beta)+b$. Multiplying the first and second equality with $1-t$ and $t$, respectively, and adding up the equations so obtained, we arrive at
\Eq{*}{
f\big(t\beta+(1-t)\alpha\big)=tf(\beta)+(1-t)f(\alpha),
}
thus we are done.
\end{proof}

\end{document}